\documentclass[12pt, reqno]{amsart}
\usepackage{amsmath, amsthm, amscd, amsfonts, amssymb, graphicx, color}
\usepackage[bookmarksnumbered, colorlinks, plainpages]{hyperref}

\newtheorem{theorem}{Theorem}[section]
\newtheorem{lemma}[theorem]{Lemma}
\newtheorem{proposition}[theorem]{Proposition}

\theoremstyle{definition}
\newtheorem{definition}[theorem]{Definition}

\theoremstyle{remark}
\newtheorem{remark}[theorem]{Remark}
\numberwithin{equation}{section}

\begin{document}


\title[Almost inner Rickart algebras]{On some counterparts of Rickart $*$-algebras}

\author[Arzikulov F.~N.]{Arzikulov F.~N.$^{1,2}$}

\address{$^{1}$ V.I. Romanovskiy Institute of Mathematics, Uzbekistan Academy of Sciences,
University 4, Olmazor, Tashkent, 100174, Uzbekistan,}

\address{$^{2}$ Andijan State University, University 129, Andijan, 170100, Uzbekistan}
\email{\textcolor[rgb]{0.00,0.00,0.84}{arzikulovfn@gmail.com}}

\author[Khakimov U.~I.]{Khakimov U.~I.$^2$}

\email{\textcolor[rgb]{0.00,0.00,0.84}{khakimov$_{-}$u@inbox.ru}}



\thanks{{\it Corresponding Author:} Arzikulov F.~N.}

\maketitle

\begin{abstract}
In the present paper, we introduce and study counterparts of Rickart involutive algebras, i.e., almost inner Rickart algebras.
We prove that a nilpotent associative algebra, which has no nilpotent elements with nonzero square roots, is an almost inner Rickart algebra.
A nilpotent associative algebra, which has no nilpotent elements with a square root $b$ such that $b^3\neq {\bf 0}$, is not an almost inner Rickart algebra if there exists a nonzero element $a$ such that $a^2\neq {\bf 0}$.

As a main result of the paper, we describe a finite-dimensional almost inner Rickart algebra $\mathcal{A}$ over a field $\mathbb{F}$, isomorphic to $\mathbb{F}^n\dot{+} \mathcal{N}$, $n=1,2$, with a nilradical $\mathcal{N}$.
Also, we classify finite-dimensional almost inner Rickart algebras over the real or complex numbers with a nonzero nilradical $\mathcal{N}$.
\end{abstract}

\medskip

\section{Introduction}

In the present paper, we introduce and study counterparts of Rickart $*$-algebras, that is almost inner Rickart algebras.
We define an almost inner Rickart algebra as an associative algebra which is a Jordan algebra with respect to the Jordan multiplication
$a\cdot b=\frac{1}{2}(ab+ba)$ close to inner RJ-algebras. Inner RJ-algebras are introduced and studied in the papers \cite{Ayupov_Arzikulov}, \cite{AA}, \cite{AKh}.

The chosen notions were built around an (inner) quadratic annihilator. For each nonempty subset $\mathcal{B}$ of an associative algebra $\mathcal{A}$, the (inner) quadratic annihilator of $\mathcal{B}$ is defined by
\[
^{\perp_q}\mathcal{B}:=\{a\in \mathcal{A}: sas={\bf 0}, \,\, \forall s\in \mathcal{B}\}.
\]
Thus, following \cite{GLPT}, an associative algebra $\mathcal{A}$ which is equal to the direct sum $\mathcal{S}\dot{+}\mathcal{N}$
of vector spaces $\mathcal{S}$ and $\mathcal{N}$, where $\mathcal{S}$ is a semisimple algebra and $\mathcal{N}$ is a nilpotent radical of $\mathcal{A}$ (including the case $\mathcal{N}\equiv\{{\bf 0}\}$),
is called an almost inner Rickart algebra if, for each element $x\in \mathcal{A}$, there exists an idempotent $e\in \mathcal{A}$ such that
$^{\perp_q}\{x\}\cap (\mathcal{S}^2\cup\mathcal{N}^2)=e\mathcal{A}e\cap (\mathcal{S}^2\cup\mathcal{N}^2)$, where $\mathcal{S}^2:=\{a^2: a\in \mathcal{S}\}$,
$\mathcal{N}^2:=\{a^2: a\in \mathcal{N}\}$.
There exist examples of almost inner Rickart algebras without unit element (cf. \cite[p. 32]{AA}). Note that, there exist pairwise non-isomorphic (associative) almost inner Rickart algebras, the Jordan algebras of which are isomorphic (see Examples in item 4) below). This is a motivation to introduce the notion of an almost inner Rickart algebra.

In the present paper, we prove that a nilpotent associative algebra $\mathcal{A}$
such that $a^2={\bf 0}$ for every element $a\in\mathcal{A}$ is an almost inner Rickart algebra.
We explain that a nilpotent associative algebra that has no nilpotent elements
with a square root $b$ such that $b^3\neq {\bf 0}$ is not an almost inner Rickart algebra
if there exists a nonzero element $a$ such that $a^2\neq {\bf 0}$.

As a main result of the paper we describe a finite-dimensional almost inner Rickart algebra $\mathcal{A}$ over a field $\mathbb{F}$, isomorphic to $\mathbb{F}^n\dot{+}\mathcal{N}$, $n=1,2$, with a nilradical $\mathcal{N}$ (Theorems  \ref{1.4} and \ref{2}).
Also, we classify finite-dimensional almost inner Rickart algebras over the real or complex numbers with a nonzero nilradical $\mathcal{N}$
(Theorem \ref{2.7}).

\section{Description of finite-dimensional almost inner Rickart algebras}

Throughout the paper let $\mathcal{A}$ be an associative algebra which is
decomposed (as a linear space) into a direct sum $\mathcal{S}\dot{+}\mathcal{N}$ of
a semisimple algebra $\mathcal{S}$ and a nilpotent radical $\mathcal{N}$ of $\mathcal{A}$.
For an associative algebra $\mathcal{A}$, we consider the following condition:

(A) for every element $x\in \mathcal{A}$ , there exists an idempotent $e\in \mathcal{A}$ such that
$$
^{\perp_q}\{x\}\cap (\mathcal{S}^2\cup\mathcal{N}^2)=e\mathcal{A}e\cap (\mathcal{S}^2\cup\mathcal{N}^2).
$$

{\bf Definition.}
An associative algebra satisfying condition (A) is called an almost inner Rickart algebra,
and condition (A) itself is called an almost inner Rickart condition.

{\bf Examples.}
1) It is obvious that every associative algebra which is an inner RJ-algebra with respect to
the multiplication $a\cdot b=\frac{1}{2}(ab+ba)$ is an almost inner Rickart algebra.

2) Let $M_n({\mathbb R})$ be the associative algebra of $n\times n$ matrices over the field ${\mathbb R}$
of real numbers, where ${\textbf{1}}$ is the unit of the algebra $M_n({\mathbb R})$, ${\{e_{i,j}\}^n_{i,j=1}}$ is the system of matrix units of the algebra $M_n({{\mathbb R}})$, $n=2k$.

a) Consider the associative algebra
$$
\mathcal{A}={\mathbb R}{\textbf{1}}+\sum_{i=1}^k {\mathbb R}e_{2i-1,2i}.
$$
This algebra is an almost inner Rickart algebra which is not involutive \cite[p. 30]{AA}.

b) Similarly, the associative algebra
$$
\mathcal{A}={\mathbb R}{\textbf{1}}+\sum_{i=1}^k {\mathbb R}e_{2i,2i-1},
$$
is an almost inner Rickart algebra.

c) Consider the associative algebra
$$
\mathcal{A}=\sum_{i=1}^k {\mathbb R}e_{2i-1,2i}
$$
This algebra is an almost inner Rickart algebra without a unit element \cite[Page 30, Example 3]{AA}. Indeed, since $\mathcal{S}^2\cup\mathcal{N}^2=\{{\bf 0}\}$, then
$$
^{\perp_q}\{{\bf 0}\}\cap (\mathcal{S}^2\cup\mathcal{N}^2)={\bf 0}(\mathcal{A}){\bf 0}\cap (\mathcal{S}^2\cup\mathcal{N}^2).
$$
Moreover, $\{e_{2i-1,2i}\}^{\perp_q} \cap (\mathcal{S}^2\cup\mathcal{N}^2)={\bf 0}(\mathcal{A}){\bf 0}\cap (\mathcal{S}^2\cup\mathcal{N}^2)$ for any $i$ from $\{1,2,\dots,k\}$.

d) Similar to Example c) the associative algebra
$$
\mathcal{A}=\sum_{i=1}^k {\mathbb R}e_{2i,2i-1}
$$
is also an almost inner Rickart algebra without a unit element.

3) A vector space $\mathcal{A}$ of dimension $n+1$
with a basis $\{e,e_1,e_2,\dots,e_n\}$ such that
\[
ee=e,\,\, ee_i=e_ie=e_i, \,\, e_i^2={\bf 0}, \,\, e_ie_j={\bf 0},  i,j=1,2,...,n.
\]
is a commutative associative algebra and, by Theorem \ref{1.4}, an almost inner Rickart algebra.

4) Let $\mathcal{A}$ be an associative algebra and $(\mathcal{A},\cdot)$ be the Jordan algebra
with respect to the multiplication $a\cdot b=\frac{1}{2}(ab+ba)$. Then,
the algebras $\mathcal{A}_2^2$, $\mathcal{A}_2^3$,
$\mathcal{A}_3^4$, $\mathcal{A}_3^6$,
$\mathcal{A}_4^8$, $\mathcal{A}_4^9$,
$\mathcal{A}_4^{19}$, $\mathcal{A}_4^{20}$ are almost inner Rickart algebras (see Tables 1 and 2) and the Jordan algebras
$(\mathcal{A}_2^2,\cdot)\cong (\mathcal{A}_2^3,\cdot)$,
$(\mathcal{A}_3^4,\cdot)\cong (\mathcal{A}_3^6,\cdot)$,
$(\mathcal{A}_4^8,\cdot)\cong (\mathcal{A}_4^9,\cdot)$,
$(\mathcal{A}_4^{19},\cdot)\cong (\mathcal{A}_4^{20},\cdot)$.
Thus, two mutually non-isomorphic almost inner Rickart algebras
give the same Jordan algebra. The examples above are a motivation to introduce the notion of an
almost inner Rickart algebra.

\begin{lemma}    \label{1.0}
Let $\mathcal{A}=\mathcal{S}\dot{+}\mathcal{N}$ be an almost inner Rickart algebra. Then
the algebra $\mathcal{A}$ has an element $e$ such that $ea^2=a^2e=a^2$ for every $a\in\mathcal{S}\cup\mathcal{N}$.
\end{lemma}

\begin{proof}
Take $x={\bf 0}$. We have $^{\perp_q}\{x\}\cap (\mathcal{S}^2\cup\mathcal{N}^2)=e\mathcal{A}e\cap (\mathcal{S}^2\cup\mathcal{N}^2)$
for some idempotent $e$ in $\mathcal{A}$. But $^{\perp_q}\{x\}\cap (\mathcal{S}^2\cup\mathcal{N}^2)=
\mathcal{A}\cap (\mathcal{S}^2\cup\mathcal{N}^2)=\mathcal{S}^2\cup\mathcal{N}^2$. Hence $e\mathcal{A}e\cap (\mathcal{S}^2\cup\mathcal{N}^2)=\mathcal{S}^2\cup\mathcal{N}^2$.
Then $ea=e(ebe)=ebe=a$ and $ae=(ebe)e=ebe=a$ for all $a\in \mathcal{S}^2\cup\mathcal{N}^2$.
\end{proof}

\begin{theorem}    \label{1.2}
A nilpotent associative algebra that has no nilpotent elements with a nonzero square root is an almost inner Rickart algebra.
\end{theorem}

\begin{proof}
Let $\mathcal{A}$ be a nilpotent associative algebra that has no nilpotent elements with a nonzero square root.
Let $a$ be a nonzero element from the algebra $\mathcal{A}$. Since $\mathcal{A}$ is nilpotent, then $a^m={\bf 0}$ for some element $m$.
Then $a^2={\bf 0}$. Otherwise, the algebra $\mathcal{A}$ contains an element $b$, which is a nilpotent element with a nonzero square root.
So, $a^2={\bf 0}$ for every non-zero element $a$ from $\mathcal{A}$. Therefore, for any $a$, $b$ from $\mathcal{A}$,
$$
ab+ba=(a+b)^2-a^2-b^2={\bf 0},\,\,\, aba+baa=aba={\bf 0}.
$$

The algebra $\mathcal{A}$ is an almost inner Rickart algebra.
Indeed, the algebra $\mathcal{A}$ does not contain a nonzero idempotent. In particular, there is no non-zero idempotent $e$ such that
$$
^{\perp_q}\{{\bf 0}\}=e\mathcal{A}e.
$$
However, since $\mathcal{S}^2\cup\mathcal{N}^2=\{{\bf 0}\}$, we have
$$
^{\perp_q}\{{\bf 0}\}\cap (\mathcal{S}^2\cup\mathcal{N}^2)={\bf 0}\mathcal{A}{\bf 0}\cap (\mathcal{S}^2\cup\mathcal{N}^2).
$$
Here $^{\perp_q}\{{\bf 0}\}=\mathcal{A}$. Similarly, for any element $a$ from $\mathcal{A}$, we have
$$
^{\perp_q}\{a\}\cap (\mathcal{S}^2\cup\mathcal{N}^2)=\{{\bf 0}\}={\bf 0}\mathcal{A}{\bf 0}\cap (\mathcal{S}^2\cup\mathcal{N}^2).
$$
The proof is complete.
\end{proof}

\begin{theorem}    \label{1.3}
Let $\mathcal{A}$ be a nilpotent associative algebra.
Then $\mathcal{A}$ is an almost inner Rickart algebra if and only if,
for any element $a$ in $\mathcal{A}$, $a^2={\bf 0}$.
\end{theorem}

\begin{proof}
Suppose that $\mathcal{A}$ is an almost inner Rickart algebra.
Let $a$ be an arbitrary nonzero element from the algebra $\mathcal{A}$ such that $a^2\neq {\bf 0}$.
Since the algebra $\mathcal{A}$ is nilpotent, $a$ is a nilpotent element.
Then $a^k\neq {\bf 0}$, $a^{k+1}={\bf 0}$ for some natural number $k\geq 2$. Then $(a^{[\frac{k}{2}]})^2\neq {\bf 0}$ and
$$
(a^{[\frac{k}{2}]})^2\in ^{\perp_q}\{a\}\cap (\mathcal{S}^2\cup\mathcal{N}^2).
$$
Hence, $^{\perp_q}\{a\}\cap (\mathcal{S}^2\cup\mathcal{N}^2)\neq\{{\bf 0}\}$. So,
there is a nonzero idempotent $e$ such that
$$
^{\perp_q}\{a\}\cap (\mathcal{S}^2\cup\mathcal{N}^2)=e\mathcal{A}e\cap (\mathcal{S}^2\cup\mathcal{N}^2).
$$
But there is no nonzero idempotent in the algebra $A$ since the algebra $\mathcal{A}$ is nilpotent.
Therefore, $\mathcal{A}$ is not an almost inner Rickart associative algebra.
This is a contradiction. So, $a^2={\bf 0}$.

Now, suppose that for any element $a$ in $\mathcal{A}$, $a^2={\bf 0}$. Then, by Theorem
\ref{1.2} $\mathcal{A}$ is an almost inner Rickart algebra. This ends the proof.
\end{proof}

\begin{lemma}    \label{1.41}
Let $\mathcal{A}$ be a finite-dimensional almost inner Rickart algebra
over a field $\mathbb{F}$ of characteristic $\neq 2$ and $\neq 3$ with a one-dimensional simple subalgebra $\mathcal{S}$ and
an $n$-dimensional commutative nilpotent radical $\mathcal{N}$ such that $\mathcal{A}=\mathcal{S}\dot{+}\mathcal{N}$, $\mathcal{N}^2\neq\{{\bf 0}\}$ and $\mathcal{N}^3=\{a^3: a\in\mathcal{N}\}=\{{\bf 0}\}$.
Then, for any idempotent $e\in\mathcal{A}$ such that
$^{\perp_q}\{x\}\cap (\mathcal{S}^2\cup\mathcal{N}^2)=e\mathcal{A}e\cap (\mathcal{S}^2\cup\mathcal{N}^2)$,
$e\in \mathcal{S}$, and, for any elements $a,b,c$ from $\mathcal{N}$,
$$
abc={\bf 0}. \eqno{(1.0)}
$$
\end{lemma}

\begin{proof}
By the condition, the associative algebra $\mathcal{N}$ has no nilpotent elements $a$
such that $a^3\neq {\bf 0}$. So, for any nonzero element $a$ from the nilradical $\mathcal{N}$ we have
$a^3={\bf 0}$. Therefore, for any $a, b$ from $\mathcal{N}$,
$$
{\bf 0}=(a+b)^3-a^3-b^3=(a^2+ab+ba+b^2)(a+b)-a^3-b^3,
$$
$$
=a^3+a^2b+aba+ab^2+ba^2+bab+b^2a+b^3-a^3-b^3=3a^2b+3b^2a,
$$
$$
{\bf 0}=(a-b)^3-a^3+b^3=(a^2-ab-ba+b^2)(a-b)-a^3+b^3,
$$
$$
=a^3-a^2b-aba+ab^2-ba^2+bab+b^2a-b^3-a^3-b^3=-3a^2b+3b^2a.
$$
Hence,
\[
a^2b={\bf 0}, b^2a={\bf 0}.
\]

As well as,
$$
abc=a\frac{1}{2}((b+c)^2-b^2-c^2)=\frac{1}{2}(a(b+c)^2-ab^2-ac^2)={\bf 0}.
$$
So, for any elements $a,b,c$ from $\mathcal{N}$, equality (1.0) is valid.

Let $a$ be an arbitrary nonzero element from the radical $\mathcal{N}$.
Then there exists a nonzero idempotent $e\in \mathcal{A}$ such that
$$
{^{\perp_q}\{a\}}\cap (\mathcal{S}^2\cup\mathcal{N}^2)=e\mathcal{A}e\cap (\mathcal{S}^2\cup\mathcal{N}^2).
$$
Due to (1.0) we have $\mathcal{N}^2\subseteq ^{\perp_q}\{a\}\cap (\mathcal{S}^2\cup\mathcal{N}^2)$. Hence,
$$
\mathcal{N}^2\subseteq e\mathcal{A}e\cap (\mathcal{S}^2\cup\mathcal{N}^2).
$$

Since $\mathcal{S}$ is a one-dimensional simple algebra, there exists an idempotent $p$ in $\mathcal{S}$
such that $\mathcal{S}=\mathbb{F}p$.
There exist $b\in\mathcal{N}$ and nonzero $x\in \mathcal{S}$ such that $e=x+b$. We have
\[
e=e^2=(x+b)^2=x^2+xb+bx+b^2, x=x^2, b=xb+bx+b^2.
\]
So, $x=p$.

If there exists $y\in\mathcal{S}^2$ such that $y\in{^{\perp_q}\{a\}}$, then
$p\in {^{\perp_q}\{a\}}$, i.e., $apa={\bf 0}$.
Hence, $p\in e(\mathcal{S}^2\cup\mathcal{N}^2)e$, i.e., $ep=pe=p$.
So, $(p+b)p=p(p+b)=p$. But $p(p+b)=(p+b)p=p+bp=p+pb$ and $pb$, $bp\in\mathcal{N}$. Hence, $bp=pb={\bf 0}$.
Therefore, $b=pb+bp+b^2=b^2$. The fact that $b^2\neq {\bf 0}$ contradicts $b\in \mathcal{N}$ and that $\mathcal{N}$ is nilpotent.
So, $b^2={\bf 0}$ and, hence, $b={\bf 0}$. From $b={\bf 0}$ it follows that $e=p$, i.e., $e\in\mathcal{S}$.

Suppose that $\mathcal{S}^2\cap {^{\perp_q}\{a\}}=\{{\bf 0}\}$. Then
$$
{^{\perp_q}\{a\}}\cap (\mathcal{S}^2\cup\mathcal{N}^2)=e\mathcal{A}e\cap (\mathcal{S}^2\cup\mathcal{N}^2)
=e\mathcal{A}e\cap \mathcal{N}^2=\mathcal{N}^2.
$$
But, for any $a\in\mathcal{N}$, we have $ea^2e=ea^2=a^2$, i.e., $(p+b)a^2=pa^2+ba^2=pa^2=a^2$ by (1.0). Similarly, $a^2p=a^2$.
Hence,
$$
e\mathcal{A}e\cap \mathcal{N}^2=\mathcal{N}^2=p\mathcal{N}^2p=p\mathcal{A}p\cap \mathcal{N}^2.
$$
Therefore, in this case we have
$$
{^{\perp_q}\{a\}}\cap (\mathcal{S}^2\cup\mathcal{N}^2)=p\mathcal{A}p\cap (\mathcal{S}^2\cup\mathcal{N}^2)
$$
and instead of $e$ we can take $p$. The proof is complete.
\end{proof}

\begin{theorem}    \label{1.4}
Let $\mathcal{A}$ be a finite-dimensional almost inner Rickart algebra
over a field $\mathbb{F}$ of characteristic $\neq 2$ and $\neq 3$ with a one-dimensional simple subalgebra $\mathcal{S}$ and
an $n$-dimensional commutative nilpotent radical $\mathcal{N}$ such that $\mathcal{A}=\mathcal{S}\dot{+}\mathcal{N}$, $\mathcal{N}^2\neq\{{\bf 0}\}$ and $\mathcal{N}^3=\{{\bf 0}\}$.
Then there is a nonzero idempotent $e\in \mathcal{A}$ such that $\mathcal{A}=\mathbb{F}e\dot{+}\mathcal{N}$ and
for any basis $\{e_1,e_2,\dots,e_n\}$ of $\mathcal{N}$ the following conditions are valid
$$
e_ie_je_k={\bf 0}, \\ \\ \ i,j,k=1,2,\ldots,n,             \eqno{(1.1)}
$$
$$
e(e_ie_j)=(e_ie_j)e=e_ie_j, \ \\ \ i,j=1,2,\ldots,n,       \eqno{(1.2)}
$$
$$
e_iee_j+e_jee_i={\bf 0},\\ \\ \ i,j=1,2,\ldots,n,                 \eqno{(1.3)}
$$
$$
ee_i\in \mathcal{N}, \,\, e_ie\in\mathcal{N}.          \eqno{(1.4)}
$$
Conversely, any associative algebra
$\mathcal{A}$ over a field $\mathbb{F}$ of characteristic $\neq 2$ and $\neq 3$
with a one-dimensional simple subalgebra $\mathcal{S}=\mathbb{F}e$, where $e$ is an idempotent element, and
an $n$-dimensional commutative nilpotent radical $\mathcal{N}$ with a basis $\{e_1,e_2,\dots,e_n\}$ such that
$\mathcal{A}=\mathcal{S}\dot{+}\mathcal{N}$, is an almost inner Rickart algebra if
the conditions (1.1)--(1.4) are valid.
\end{theorem}

\begin{proof}
Since $\mathcal{N}^3=\{{\bf 0}\}$, we have,
by Lemma  \ref{1.41}, for any elements $a,b,c$ from $\mathcal{N}$,
$$
abc={\bf 0}. \eqno{(1.5)}
$$

Since $\mathcal{N}^2\neq\{{\bf 0}\}$, there exists $x\in \mathcal{N}$ such that $x^2\neq {\bf 0}$.
Let $a$ be an arbitrary nonzero element from the radical $\mathcal{N}$.
Then, by Lemma \ref{1.41} and (1.5), there exists a nonzero idempotent $e\in \mathcal{S}$ such that
$$
x^2\in^{\perp_q}\{a\}\cap (\mathcal{S}^2\cup\mathcal{N}^2)=e\mathcal{A}e\cap (\mathcal{S}^2\cup\mathcal{N}^2).
$$
We also have $e\in e\mathcal{A}e$. Since $e$ is idempotent, we have
$e\in\mathcal{S}^2$. Hence, $e\in e\mathcal{A}e\cap (\mathcal{S}^2\cup\mathcal{N}^2)$, i.e.,
$e\in ^{\perp_q}\{a\}\cap (\mathcal{S}^2\cup\mathcal{N}^2)$. From this it follows that $aea={\bf 0}$.

Due to (1.5) we have $\mathcal{N}^2\subseteq ^{\perp_q}\{a\}\cap (\mathcal{S}^2\cup\mathcal{N}^2)$. Hence,
$$
\mathcal{N}^2\subseteq e\mathcal{A}e\cap (\mathcal{S}^2\cup\mathcal{N}^2).
$$

Let $b$ be another nonzero element from the nilradical $\mathcal{N}$.
Then there exists a nonzero idempotent $f\in \mathcal{S}$ such that
$$
x^2\in^{\perp_q}\{b\}\cap (\mathcal{S}^2\cup\mathcal{N}^2)=f\mathcal{A}f\cap (\mathcal{S}^2\cup\mathcal{N}^2)
$$
and
$$
\mathcal{N}^2\subseteq f\mathcal{A}f\cap (\mathcal{S}^2\cup\mathcal{N}^2)
$$
Thus, $ea^2=fa^2=a^2$, $eb^2=fb^2=b^2$. Furthermore,
$$
ec=fc=c, \,\, \, c\in \mathcal{N}^2.
$$
Since codimension of $\mathcal{N}$ equals $1$ we have $e=f$.
If we take $a+b$ then
$$
a^2,\,\, b^2,\,\, (a+b)^2\in^{\perp_q}\{a\}\cap (\mathcal{S}^2\cup\mathcal{N}^2)=e\mathcal{A}e\cap (\mathcal{S}^2\cup\mathcal{N}^2),
$$
$$
(a+b)^2=e(a+b)^2=a^2+2e(ab)+b^2,
$$
$$
(a+b)^2=(a+b)^2e=a^2+2(ab)e+b^2.
$$
Hence,
$$
e(ab)=(ab)e=ab.                \eqno{(1.6)}
$$

Let ${e_1,e_2,\dots ,e_n}$ be a basis in $\mathcal{N}$. Then, by (1.5)
we have (1.1).  Also, by (1.6), we have (1.2).
At the same time,
$$
e_iee_j+e_jee_i={\bf 0}, \,\, i,j=1,2,\ldots,n,
$$
i.e. (1.3) is valid. Indeed, since $aea={\bf 0}$, $a\in \mathcal{N}$, we have, for any $i$, $j$, $(e_i+e_j)e(e_i+e_j)={\bf 0}$.
Hence,
\[
e_iee_j+e_jee_i={\bf 0}.
\]
So, (1.3) is valid.
Since $\mathcal{N}$ is an ideal, we have (1.4) is also valid.

Now let us prove that the associative algebra $\mathcal{A}$ with a basis $\{e,e_1,e_2,\dots,e_n\}$,
satisfying the properties (1.1)--(1.4), is an almost inner Rickart algebra. We have
$$
^{\perp_q}\{e\}\cap (\mathcal{S}^2\cup\mathcal{N}^2)=\{{\bf 0}\}={\bf 0}\mathcal{A}{\bf 0}\cap (\mathcal{S}^2\cup\mathcal{N}^2),
$$
since $(\mathcal{S}^2\cup\mathcal{N}^2)\subseteq e\mathcal{A}e$ by (1.2).

We take an arbitrary element of the form
$$
a=\sum_{i=1}^n \lambda_ie_i.
$$
By equalities (1.1) and (1.3), we have
$$
e\in^{\perp_q}\{a\}\cap (\mathcal{S}^2\cup\mathcal{N}^2)=\mathcal{S}^2\cup\mathcal{N}^2=e\mathcal{A}e\cap (\mathcal{S}^2\cup\mathcal{N}^2).
$$

Now we take an element of the form $\lambda e+a$, $\lambda\neq 0$. Since $\mathcal{N}$ is an ideal, by (1.2), we have
$$
(\lambda e+a)b^2(\lambda e+a)=\lambda eb^2\lambda e+\lambda eb^2a+ab^2\lambda e+ab^2a=\lambda eb^2\lambda e=\lambda^2b^2\neq {\bf 0}
$$
for every $b\in \mathcal{N}$ such that $b^2\neq {\bf 0}$. We also have
$$
(\lambda e+a)\mu^2 e(\lambda e+a)=\lambda e\mu^2e\lambda e+c=\lambda^2\mu^2e+c,
$$
where
\[
c=2[\lambda e\mu^2ea+a\mu^2ea+a\mu^2e(\lambda e+a)]
\]
and $c\in \mathcal{N}$ since $\mathcal{N}$ is an ideal. Therefore $(\lambda e+a)(\mu e)^2(\lambda e+a)\neq {\bf 0}$.
Hence,
$$
^{\perp_q}\{\lambda e+a\}\cap (\mathcal{S}^2\cup\mathcal{N}^2)=\{{\bf 0}\}={\bf 0}\mathcal{A}{\bf 0}\cap (\mathcal{S}^2\cup\mathcal{N}^2).
$$
So, $\mathcal{A}$ is an almost inner Rickart algebra. The proof is complete.
\end{proof}

It should be noted that Theorem 2.4 in \cite{AKh}, which is a Jordan analog of Theorem \ref{1.4}, is incorrect.
A correct version of this theorem is proved in \cite{AKh2}. Note that Theorem 2.4 in the paper \cite{AKh} is valid
for a so-called almost inner RJ-algebra, defined following way:

\begin{definition}
A Jordan algebra $\mathcal{J}$ with a semisimple subalgebra $\mathcal{S}$
and a nilpotent radical $\mathcal{N}$ such that $\mathcal{J}=\mathcal{S}\dot{+}\mathcal{N}$ is called
an almost inner RJ-algebra if, for any element $a\in \mathcal{J}$, there exists an idempotent $e\in\mathcal{S}$ such that
\[
{^{\perp_q}\{a\}}\cap (\mathcal{S}^2\cup\mathcal{N}^2)=U_e(\mathcal{J})\cap (\mathcal{S}^2\cup\mathcal{N}^2).
\]
\end{definition}

It should be noted that this and all other statements of the paper \cite{AKh} are valid
for almost inner RJ-algebras. In \cite{AKh2} corrected versions for inner RJ-algebras of the statements
in \cite{AKh} are given.

\begin{remark}
Note that the associative algebra, indicated in theorem \ref{1.4} is a direct sum of an indecomposable non-nilpotent algebra
and a nilpotent algebra.
If, for example, the nilradical $\mathcal{N}$ of an almost inner Rickart algebra $\mathcal{A}$ is of codimension 2,
then the quantity of non-nilpotent indecomposable algebras in the direct sum of $\mathcal{A}$ may be one or two.
For example, the almost inner Rickart algebra
\[
\mathcal{A}_1: e_1^2=e_1, e_2^2=e_2, e_1n_1=n_1, n_2e_2=n_2,
\]
where $\{e_1, e_2, n_1, n_2\}$ is a basis of the algebra $\mathcal{A}_1$ and the omitted products vanish,
is decomposable. The subalgebras
\[
\mathcal{B}_1: e_1^2=e_1, e_1n_1=n_1, \mathcal{B}_2: e_2^2=e_2, e_2n_2=n_2
\]
are almost inner Rickart algebras satisfying the conditions of theorem \ref{1.4} and
$\mathcal{A}_1=\mathcal{B}_1\oplus\mathcal{B}_2$.
The associative algebra
\[
\mathcal{A}_2: e_1^2=e_1, e_2^2=e_2, n_1e_1=n_1, e_2n_1=n_1, n_2e_1=n_2
\]
(\cite[Table 2.3, algebra $As^{19}_4$]{RRB}), where $\{e_1, e_2, n_1, n_2\}$ is a basis of the algebra $\mathcal{A}_2$ and the omitted products vanish,
is indecomposable. The nilradical $\mathcal{N}=<n_1,n_2>$ is an almost inner Rickart algebra by Theorem \ref{1.3}. Then
$\mathcal{A}_2$ is also an almost inner Rickart algebra by Theorem \ref{2}.
The subalgebras
\[
\mathcal{C}_1: e_1^2=e_1, n_1e_1=n_1, n_2e_1=n_2,  \mathcal{C}_2: e_2^2=e_2, e_2n_1=n_1
\]
of the algebra $\mathcal{A}_2$ are almost inner Rickart algebras satisfying the conditions of Theorem \ref{1.4}.
\end{remark}

Let $\mathcal{A}$ be a finite-dimensional almost inner Rickart algebra which is
decomposed (as a linear space) into a direct sum $\mathcal{S}\dot{+}\mathcal{N}$ of
a semisimple algebra $\mathcal{S}$ and a nilpotent radical $\mathcal{N}$ of $\mathcal{A}$.
For $\mathcal{A}$, we consider the following condition:

(B) for any idempotent $e\in\mathcal{A}$ such that
$^{\perp_q}\{x\}\cap (\mathcal{S}^2\cup\mathcal{N}^2)=e\mathcal{A}e\cap (\mathcal{S}^2\cup\mathcal{N}^2)$,
$e\in \mathcal{S}$.

\begin{lemma} \label{2.71}
Let $\mathcal{A}$ be a finite-dimensional associative algebra over the real or complex numbers which is
decomposed (as a linear space) into a direct sum $\mathcal{S}\dot{+}\mathcal{N}$ of
a semisimple algebra $\mathcal{S}$ and a nilpotent radical $\mathcal{N}$ of $\mathcal{A}$.
Then, if $\mathcal{A}$ is an almost inner Rickart algebra satisfying Condition (B), then
$\mathcal{S}$ is also an almost inner Rickart algebra.
\end{lemma}

\begin{proof}
Let $a$ be an arbitrary nonzero element from $\mathcal{S}$.
Then there exists a nonzero idempotent $e\in \mathcal{S}$ such that
$$
{^{\perp_q}\{a\}}\cap (\mathcal{S}^2\cup\mathcal{N}^2)=e\mathcal{A}e\cap (\mathcal{S}^2\cup\mathcal{N}^2).
$$
Hence,
$$
{^{\perp_q}\{a\}}\cap \mathcal{S}^2=e\mathcal{A}e\cap \mathcal{S}^2=e\mathcal{S}e\cap \mathcal{S}^2.
$$
This ends the proof.
\end{proof}

\begin{lemma}    \label{1.1}
Let $\mathcal{A}=\mathcal{S}\dot{+}\mathcal{N}$ be an almost inner Rickart algebra satisfying Condition (B). Then
the algebra $\mathcal{A}$ has no a nilpotent element $a$ in $\mathcal{S}\cup\mathcal{N}$ such that $a^3\neq {\bf 0}$.
\end{lemma}

\begin{proof}
Take an element $a\in \mathcal{S}\cup\mathcal{N}$. Suppose that $a^4={\bf 0}$.
Then $a^2\in{^{\perp_q}\{a\}}\cap (\mathcal{S}^2\cup\mathcal{N}^2)$. Since $\mathcal{A}$ is an almost inner Rickart algebra,
there exists a nonzero idempotent $e\in \mathcal{A}$ such that $^{\perp_q}\{a\}\cap (\mathcal{S}^2\cup\mathcal{N}^2)
=e\mathcal{A}e\cap (\mathcal{S}^2\cup\mathcal{N}^2)$. Therefore $aea={\bf 0}$ and $ea^2=a^2$ since $e\in e\mathcal{A}e\cap (\mathcal{S}^2\cup\mathcal{N}^2)$.
Hence, $a^3=aa^2=aea^2=(aea)a={\bf 0}$.

Further, we shall apply the induction.
The induction assumption is $(a^2)^k={\bf 0}$,  $1\leq k\leq n$, implies that $a^3={\bf 0}$, where the induction step
is $n$. If $n=1$, $2$ then the induction assumption is valid by the previous paragraph.

Suppose that $(a^2)^{n+1}={\bf 0}$. Then, if $n=2k$, then $(a^2)^{n+1}a^2=(a^2)^{2k+2}=(a^4)^{k+1}={\bf 0}$ and $1\leq k+1<n$.
By the inductive assumption $a^6={\bf 0}$. So, $(a^2)a^2(a^2)={\bf 0}$ and $a^2\in{^{\perp_q}\{a^2\}}\cap (\mathcal{S}^2\cup\mathcal{N}^2)$. Since $\mathcal{A}$ is an almost inner Rickart algebra,
there exists an idempotent $f\in \mathcal{A}$ such that $^{\perp_q}\{a^2\}\cap (\mathcal{S}^2\cup\mathcal{N}^2)
=f\mathcal{A}f\cap (\mathcal{S}^2\cup\mathcal{N}^2)$. Therefore $fff=f\in f\mathcal{A}f\cap (\mathcal{S}^2\cup\mathcal{N}^2)$, $a^2fa^2={\bf 0}$ and $fa^2=a^2$.
Hence, $a^4=a^2fa^2={\bf 0}$. By the previous paragraph $a^3={\bf 0}$.

If $n=2k+1$, then $(a^2)^{n+1}=(a^2)^{2k+2}=(a^4)^{k+1}={\bf 0}$ and $1\leq k+1<n$.
Again, by the inductive assumption $a^6={\bf 0}$. Hence, by the previous paragraph $a^3={\bf 0}$.
Thus, by the induction, we obtain that
for  $a\in \mathcal{S}\cup\mathcal{N}$, if $a^n={\bf 0}$ for some natural number $n\geq 4$, then $a^3={\bf 0}$.
The proof is complete.
\end{proof}

\begin{theorem} \label{2.7}
Let $\mathcal{A}$ be a finite-dimensional almost inner Rickart algebra over the real or complex numbers satisfying Condition (B),
with a nilpotent radical $\mathcal{N}il(\mathcal{A})$.
Then there exist pairwise orthogonal idempotents $e_1,e_2,\dots,e_m$ in $\mathcal{A}$ such that
\[
\mathcal{A}=(e_1\mathbb{F}\oplus e_2\mathbb{F}\oplus\dots\oplus e_m\mathbb{F})\dot{+}\mathcal{N}il(\mathcal{A}),
\]
where $\mathbb{F}=\mathbb{R}$, $\mathbb{C}$, $\mathbb{H}$.
\end{theorem}

\begin{proof}
By the Wedderburn-Malcev and Cartan-Frobenius-Molien theorems, every finite-dimensional associative algebra $\mathcal{A}$ over the real or complex numbers
can be represented as $\mathcal{A}=\mathcal{S}\dot{+}\mathcal{N}$, where $\mathcal{N}$ is a nil-radical and $\mathcal{S}$ is
a semi-simple algebra. In its turn $\mathcal{S}$ can be represented as
\[
\mathcal{S}=\mathcal{C}_1\oplus\mathcal{C}_2\oplus\dots\mathcal{C}_n,
\]
where $\mathcal{C}_i$, $i=1,2,\dots n$, are simple algebras and $\mathcal{C}_i\cong M_{n_i}(\mathcal{N}_i)$,
where $M_{n_i}(\mathcal{N}_i)$ is the algebra of $n_i\times n_i$ matrices with entries from a division algebra $\mathcal{N}_i$.
Also, by Lemma \ref{1.1} $\mathcal{N}il(\mathcal{A})^3=\{{\bf 0}\}$.
Therefore, by the examples in item 2) of Remark 1.6 in \cite{AKh},
every finite-dimensional almost inner Rickart algebra $\mathcal{A}$ has the following form
\[
\mathcal{A}=\mathcal{S}\dot{+}\mathcal{N}il(\mathcal{A}),
\]
where $\mathcal{N}il(\mathcal{A})$ is the nilpotent radical of $\mathcal{A}$ and
$\mathcal{S}$ is a semisimple subalgebra every simple subalgebra of which is
isomorphic to only one of the following algebras: an abelian algebra, the algebra of
$2\times 2$-matrices over a division algebra, the algebra of
$3\times 3$-matrices over a division algebra.

By Lemma \ref{2.71} $\mathcal{S}$ is also an almost inner Rickart algebra.
Moreover, every direct summand $\mathcal{C}_i$ is also an almost inner Rickart algebra.
Indeed, let $e_i$ be a unit element and $x$ is an element of $\mathcal{C}_i$. Then, there exists an idempotent
$e$ in $\mathcal{S}$ such that $^{\perp_q}\{x\}\cap \mathcal{S}^2=e\mathcal{S}e\cap \mathcal{S}^2$.
Hence,
$(^{\perp_q}\{x\}\cap \mathcal{S}^2)\cap\mathcal{C}_i=e_i\mathcal{S}e_i\cap {\mathcal{C}_i}^2=e_i\mathcal{C}_ie_i\cap \mathcal{C}_i^2$.

By the examples on page 35 of \cite{AA} and on page 8 of \cite{AKh} the
algebras $M_2(\mathbb{F})$, $M_3(\mathbb{F})$, where $\mathbb{F}=\mathbb{R}, \mathbb{C}, \mathbb{H}$,
are not inner RJ-algebras. But for a simple associative algebra the definition of
an almost inner Rickart algebra coincides with the definition of an inner RJ-algebra.
So, the algebras $M_2(\mathbb{F})$, $M_3(\mathbb{F})$, where $\mathbb{F}=\mathbb{R}, \mathbb{C}, \mathbb{H}$,
are not almost inner Rickart algebras. Therefore, there exist pairwise orthogonal idempotents
$e_1,e_2,\dots,e_m$ in $\mathcal{A}$ such that
\[
\mathcal{S}=e_1\mathbb{F}\oplus e_2\mathbb{F}\oplus\dots\oplus e_m\mathbb{F}.
\]
The proof is complete.
\end{proof}

\begin{theorem}    \label{2}
Let $\mathcal{A}$ be an indecomposible finite-dimensional almost inner Rickart algebra
over the real or complex numbers $\mathbb{F}$ satisfying Condition (B)
with a two-dimensional semisimple subalgebra $\mathcal{S}$ and
an $n$-dimensional commutative nilpotent radical $\mathcal{N}$ such that $\mathcal{A}=\mathcal{S}\dot{+}\mathcal{N}$, $\mathcal{N}^2\neq\{{\bf 0}\}$.
Then there exist two mutually orthogonal nonzero idempotents $p_1$, $p_2\in \mathcal{A}$ such that
$\mathcal{A}=(\mathbb{F}p_1\oplus\mathbb{F}p_2)\dot{+}\mathcal{N}$ and
for any basis $\{e_1,e_2,\dots,e_n\}$ of $\mathcal{N}$ one of the following two cases is valid

{\bf the first case}
$$
e_ie_je_k={\bf 0}, \\ \\ \ i,j,k=1,2,\ldots,n,                     \eqno{(1.1)}
$$
$$
p_1(e_ie_j)=(e_ie_j)p_1=e_ie_j, \ \\ i,j=1,2,\ldots,n,       \eqno{(1.2)}
$$
$$
e_ip_1e_j+e_jp_1e_i={\bf 0},\\ \\ i,j=1,2,\ldots,n,                 \eqno{(1.3)}
$$
$$
p_ke_i\in \mathcal{N},e_ip_k\in\mathcal{N}, k=1,2,            \eqno{(1.4)}
$$
$$
p_2(e_ie_j+e_je_i)={\bf 0}, i,j=1,2,\dots,n,                        \eqno{(1.5)}
$$

{\bf the second case}
$$
e_ie_je_k={\bf 0}, \\ \\ \ i,j,k=1,2,\ldots,n,             \eqno{(2.1)}
$$
$$
(p_1+p_2)(e_ie_j)=(e_ie_j)(p_1+p_2)=e_ie_j, \ \\ \ i,j=1,2,\ldots,n,       \eqno{(2.2)}
$$
$$
e_i(p_1+p_2)e_j+e_j(p_1+p_2)e_i={\bf 0},\\ \\ \ i,j=1,2,\ldots,n,                 \eqno{(2.3)}
$$
$$
(p_1+p_2)e_i\in \mathcal{N}, e_i(p_1+p_2)\in\mathcal{N},          \eqno{(2.4)}
$$
$$
\forall x\in \mathcal{N},
\text{ if } p_1x^2p_1=0 \text{ or } p_2x^2p_2=0 \text{ then } p_1x^2p_2+p_2x^2p_1=0,          \eqno{(2.5)}
$$
$$
e_ip_ke_j+e_jp_ke_i={\bf 0}, \,\,\, k=1,2, \,\,\, i,j=1,2,\ldots,n.          \eqno{(2.6)}
$$
Conversely, any associative algebra
$\mathcal{A}$ over a field $\mathbb{F}$ of characteristic $\neq 2$ and $\neq 3$
with a two-dimensional semisimple subalgebra $\mathcal{S}=\mathbb{F}p_1\oplus\mathbb{F}p_2$, where $p_1$, $p_2$ are
mutually orthogonal idempotents, and
an $n$-dimensional commutative nilpotent radical $\mathcal{N}$ with a basis $\{e_1,e_2,\dots,e_n\}$ such that
$\mathcal{A}=\mathcal{S}\dot{+}\mathcal{N}$, is an almost inner Rickart algebra if
conditions (1.1)--(1.5) or conditions (2.1)--(2.5) are valid.
\end{theorem}

\begin{proof}
By Theorem \ref{2.7}, there exist pairwise orthogonal idempotents $p_1,p_2$ in $\mathcal{A}$ such that
\[
\mathcal{A}=(p_1\mathbb{F}\oplus p_2\mathbb{F})\dot{+}\mathcal{N}.
\]

Since $\mathcal{A}$ satisfies Condition (B), we have, by Lemma \ref{1.1}, $\mathcal{N}^3=\{{\bf 0}\}$.
Therefore, similar to the proof of Lemma \ref{1.41}, we have, for any elements $a$, $b$, $c$ from $\mathcal{N}$,
\[
a(bc)=(ab)c={\bf 0}.    \eqno{(1.6)}
\]
From this it follows that (1.1).

Since $\mathcal{N}^2\neq\{{\bf 0}\}$, there exists $x\in \mathcal{N}$ such that $x^2\neq {\bf 0}$.
Let $a$ be an arbitrary nonzero element from the radical $\mathcal{N}$.
Then, by (1.6), there exists a nonzero idempotent $e\in \mathcal{A}$ such that
$$
x^2\in^{\perp_q}\{a\}\cap (\mathcal{S}^2\cup\mathcal{N}^2)=e\mathcal{A}e\cap (\mathcal{S}^2\cup\mathcal{N}^2).
$$
We also have $e\in e\mathcal{A}e$ and, by Condition (B), $e\in\mathcal{S}$.
Hence, $e\in e\mathcal{A}e\cap (\mathcal{S}^2\cup\mathcal{N}^2)$, i.e.,
$e\in ^{\perp_q}\{a\}\cap (\mathcal{S}^2\cup\mathcal{N}^2)$. From this it follows that $aea={\bf 0}$.

Due to (1.6) we have $\mathcal{N}^2\subseteq ^{\perp_q}\{a\}\cap (\mathcal{S}^2\cup\mathcal{N}^2)=e\mathcal{A}e\cap (\mathcal{S}^2\cup\mathcal{N}^2)$. Hence,
$$
\mathcal{N}^2\subseteq e\mathcal{A}e\cap (\mathcal{S}^2\cup\mathcal{N}^2).      \eqno{(1.7)}
$$

There are following three possible cases
\[
\left\{
  \begin{array}{ll}
    e=p_1, & \hbox{} \\
    e=p_2, & \hbox{} \\
    e=p_1+p_2. & \hbox{}
  \end{array}
\right.
\]

Now, suppose that $e=p_1$, i.e.,
\[
{^{\perp_q}\{a\}}\cap (\mathcal{S}^2\cup\mathcal{N}^2)=p_1\mathcal{A}p_1\cap (\mathcal{S}^2\cup\mathcal{N}^2).
\]

We prove condition (1.2). If we take arbitrary elements $x$, $y\in \mathcal{N}$, then, by (1.7),
$$
x^2,\,\, y^2,\,\, (x+y)^2\in^{\perp_q}\{a\}\cap (\mathcal{S}^2\cup\mathcal{N}^2)=p_1\mathcal{A}p_1\cap (\mathcal{S}^2\cup\mathcal{N}^2),
$$
$$
(x+y)^2=p_1(x+y)^2=x^2+2p_1(xy)+y^2,
$$
$$
(x+y)^2=(x+y)^2p_1=x^2+2(xy)p_1+y^2.
$$
Hence, $p_1(xy)=(xy)p_1=xy$, and, from this it follows that (1.2) is valid.

We prove (1.3). Let $b$ be another nonzero element of the radical $\mathcal{N}$.
Then there exists a nonzero idempotent $f\in \mathbb{F}p_1\oplus \mathbb{F}p_2$ such that
$$
a^2\in^{\perp_q}\{b\}\cap (\mathcal{S}^2\cup\mathcal{N}^2)=f\mathcal{A}f\cap (\mathcal{S}^2\cup\mathcal{N}^2), \,\, \, bfb={\bf 0}.
$$
Due to (1.6) we have $\mathcal{N}^2\subseteq ^{\perp_q}\{b\}\cap (\mathcal{S}^2\cup\mathcal{N}^2)=f\mathcal{A}f\cap (\mathcal{S}^2\cup\mathcal{N}^2)$.
We also have $\mathcal{N}^2\subseteq p_1\mathcal{A}p_1\cap (\mathcal{S}^2\cup\mathcal{N}^2)$. Hence,
$f\neq p_2$ and, $f=p_1+p_2$ or $f=p_1$. Therefore, $p_1\in f\mathcal{A}f\cap (\mathcal{S}^2\cup\mathcal{N}^2)$,
i.e., $p_1\in^{\perp_q}\{b\}$. Hence, $bp_1b={\bf 0}$. Since $b$ is arbitrarily chosen,
for any $i$, $j$, $(e_i+e_j)e(e_i+e_j)={\bf 0}$.
Hence, $e_iee_j+e_jee_i={\bf 0}$. So, (1.3) is valid.

Since $\mathcal{N}$ is an ideal, we have (1.4) is also valid.

We consider the annihilator ${^{\perp_q}\{p_2\}}$. It is clear that
$p_1\in{^{\perp_q}\{p_2\}}\cap (\mathcal{S}^2\cup\mathcal{N}^2)$ since $p_1\cdot p_2={\bf 0}$ and $p_2p_1p_2={\bf 0}$.
Also, from
\[
p_2(p_1\mathcal{A}p_1)p_2=(p_2p_1)\mathcal{A}(p_2p_1)={\bf 0}
\]
it follows that
\[
\mathcal{N}^2\subset p_1\mathcal{A}p_1\cap (\mathcal{S}^2\cup\mathcal{N}^2)\subset {^{\perp_q}\{p_2\}}\cap (\mathcal{S}^2\cup\mathcal{N}^2).
\]
Hence, $p_2a^2p_2={\bf 0}$ for all $a\in \mathcal{N}$.
From $p_2((a+b)^2-a^2-b^2)p_2={\bf 0}$ for all $a$, $b\in \mathcal{N}$ it follows that $p_2(ab+ba)p_2={\bf 0}$.
Note that, since $\mathcal{A}$ is an almost inner Rickart algebra satisfying condition (B) and
$p_1\in {^{\perp_q}\{p_2\}}\cap (\mathcal{S}^2\cup\mathcal{N}^2)=e\mathcal{A}e\cap (\mathcal{S}^2\cup\mathcal{N}^2)$
for some idempotent $e\neq {\bf 0}$ in $\mathcal{A}$ we have $e\in\mathcal{S}^2$ and $p_2ep_2={\bf 0}$. Hence, $p_2e={\bf 0}$. So, $e=p_1$ and
\[
{^{\perp_q}\{p_2\}}\cap (\mathcal{S}^2\cup\mathcal{N}^2)=p_1\mathcal{A}p_1\cap (\mathcal{S}^2\cup\mathcal{N}^2).
\]

Now, we prove (1.5). We take the annihilator ${^{\perp_q}\{p_1\}}$.
It is clear that
$p_2\in{^{\perp_q}\{p_1\}}\cap (\mathcal{S}^2\cup\mathcal{N}^2)$ since $p_1\cdot p_2={\bf 0}$ and $p_1p_2p_1={\bf 0}$.
Also, from
\[
\mathcal{N}^2\subset p_1\mathcal{A}p_1
\]
it follows that $p_2(\mathcal{N}^2)={\bf 0}$, $p_2a^2={\bf 0}$ for any $a\in \mathcal{N}$.
From this it follows that $p_2(ab+ba)={\bf 0}$ for any $a$, $b\in \mathcal{N}$.
In particular,
\[
p_2(e_ie_j+e_je_i)={\bf 0}, i,j=1,2,\dots,n.  \eqno{(1.5)}
\]
This ends the case $e=p_1$.

Now let us prove that
an associative algebra
$\mathcal{A}$ over a field $\mathbb{F}$ of characteristic $\neq 2$ and $\neq 3$
with a two-dimensional semisimple subalgebra $\mathcal{S}=\mathbb{F}p_1\oplus\mathbb{F}p_2$, where $p_1$, $p_2$ are
mutually orthogonal idempotents, and
an $n$-dimensional commutative nilpotent radical $\mathcal{N}$ with a basis $\{e_1,e_2,\dots,e_n\}$ such that
$\mathcal{A}=\mathcal{S}\dot{+}\mathcal{N}$, is an almost inner Rickart algebra if
conditions (1.1)--(1.5) are valid.

We take the annihilator $^{\perp_q}\{p_1\}$. It is clear that $p_2\in ^{\perp_q}\{p_1\}$.
We prove that
$$
^{\perp_q}\{p_1\}\cap (\mathcal{S}^2\cup\mathcal{N}^2)=p_2\mathcal{A}p_2\cap (\mathcal{S}^2\cup\mathcal{N}^2).
$$

Indeed, we have $p_2\mathcal{A}p_2\subseteq ^{\perp_q}\{p_1\}$. The equality
\[
p_1a^2p_1=a^2,
\]
where $a\in \mathcal{N}$, is valid by (1.2).
Hence, $^{\perp_q}\{p_1\}\cap (\mathcal{S}^2\cup\mathcal{N}^2)=\{\mu^2p_2:\mu\in \mathbb{F}\}$. At the same time, by (1.5),
$p_2a^2={\bf 0}$ for any $a\in \mathcal{N}$.
Hence, $p_2\mathcal{A}p_2\cap (\mathcal{S}^2\cup\mathcal{N}^2)=\{\mu^2p_2:\mu\in \mathbb{F}\}$.
So, $^{\perp_q}\{p_1\}\cap (\mathcal{S}^2\cup\mathcal{N}^2)=p_2\mathcal{A}p_2\cap (\mathcal{S}^2\cup\mathcal{N}^2)$.

Now we take the annihilator $^{\perp_q}\{p_2\}$. It is clear that $p_1\in ^{\perp_q}\{p_2\}$.
We prove that
$$
^{\perp_q}\{p_2\}\cap (\mathcal{S}^2\cup\mathcal{N}^2)=p_1\mathcal{A}p_1\cap (\mathcal{S}^2\cup\mathcal{N}^2).
$$
Indeed, we have $p_1\mathcal{A}p_1\cap (\mathcal{S}^2\cup\mathcal{N}^2)\subseteq ^{\perp_q}\{p_2\}\cap (\mathcal{S}^2\cup\mathcal{N}^2)$.
Clearly $^{\perp_q}\{p_2\}\cap \mathcal{S}^2\subseteq p_1\mathcal{A}p_1\cap \mathcal{S}^2$.
It is sufficient to prove $^{\perp_q}\{p_2\}\cap \mathcal{N}^2\subseteq p_1\mathcal{A}p_1\cap \mathcal{N}^2$.
Take an arbitrary element $x^2\in^{\perp_q}\{p_2\}\cap \mathcal{N}^2$. Then $x^2\in p_1\mathcal{A}p_1$ by (1.2).
Hence, $x^2\in p_1\mathcal{A}p_1\cap \mathcal{N}^2$. So, $^{\perp_q}\{p_2\}\cap \mathcal{N}^2\subseteq p_1\mathcal{A}p_1\cap \mathcal{N}^2$.
Therefore, $^{\perp_q}\{p_2\}\cap (\mathcal{S}^2\cup\mathcal{N}^2)=p_1\mathcal{A}p_1\cap (\mathcal{S}^2\cup\mathcal{N}^2)$.

We take an element of the form $\lambda p_1+a$, $\lambda\neq 0$, and consider
$$
^{\perp_q}\{\lambda p_1+a\}\cap (\mathcal{S}^2\cup\mathcal{N}^2).
$$
Then, for any $b\in\mathcal{N}$ such that $b^2\neq 0$,
\[
(\lambda p_1+a)b^2(\lambda p_1+a)=\lambda^2 p_1b^2p_1=\lambda^2b^2\neq {\bf 0}
\]
by (1.2). Hence, $^{\perp_q}\{\lambda p_1+a\}\cap \mathcal{N}^2=\{{\bf 0}\}$.  Also
\[
(\lambda p_1+a)p_2(\lambda p_1+a)=\lambda^2 p_1p_2p_1+\lambda p_1p_2a+\lambda ap_2p_1+ap_2a=ap_2a.
\]
If $ap_2a={\bf 0}$, then $p_2\in^{\perp_q}\{\lambda p_1+a\}\cap (\mathcal{S}^2\cup\mathcal{N}^2)$.
Hence,
$$
^{\perp_q}\{\lambda p_1+a\}\cap (\mathcal{S}^2\cup\mathcal{N}^2)
=p_2\mathcal{A}p_2\cap (\mathcal{S}^2\cup\mathcal{N}^2)=\{\mu^2p_2: \mu\in\mathbb{F}\}.
$$
Else, if $ap_2a\neq {\bf 0}$, then $p_2\notin ^{\perp_q}\{\lambda p_1+a\}\cap (\mathcal{S}^2\cup\mathcal{N}^2)$.
Therefore,
\[
^{\perp_q}\{\lambda p_1+a\}\cap (\mathcal{S}^2\cup\mathcal{N}^2)=\{{\bf 0}\}=0A0\cap (\mathcal{S}^2\cup\mathcal{N}^2).
\]

Now we take an element of the form $\lambda p_1+\mu p_2+a$, $\lambda\neq 0$, $\mu\neq 0$ and
consider the annihilator $^{\perp_q}\{\lambda p_1+\mu p_2+a\}$. Since $\mathcal{N}$ is an ideal,
by the properties of the Peirce decomposition and (1.2), (1.5), we have
$$
(\lambda p_1+\mu p_2+a)b^2(\lambda p_1+\mu p_2 +a)
$$
$$
=(\lambda p_1+\mu p_2)b^2(\lambda p_1+\mu p_2)+
(\lambda p_1+\mu p_2)b^2a+ab^2(\lambda p_1+\mu p_2)+ab^2a
$$
$$
=(\lambda p_1+\mu p_2)b^2(\lambda p_1+\mu p_2)=\lambda^2 p_1b^2p_1=\lambda^2b^2\neq {\bf 0}
$$
if $b^2\neq {\bf 0}$, for every $b\in \mathcal{N}$.
We also have
$$
(\lambda p_1+\mu p_2+a)(\nu p_1+\rho p_2)^2(\lambda p_1+\mu p_2 +a)
$$
$$
=(\lambda p_1+\mu p_2+a)(\nu^2 p_1+\rho^2 p_2)^2(\lambda p_1+\mu p_2+a)
$$
$$
=\lambda^2\nu^2 p_1+\mu^2\rho^2 p_2+c,
$$
where $c\in \mathcal{N}$ since $\mathcal{N}$ is an ideal. Therefore
$$
(\lambda p_1+\mu p_2+a)(\nu p_1+\rho p_2)^2(\lambda p_1+\mu p_2 e+a)\neq {\bf 0}
$$
if $\nu\neq 0$ or $\rho\neq 0$. Hence,
$$
^{\perp_q}\{\lambda p_1+\mu p_2+a\}\cap (\mathcal{S}^2\cup\mathcal{N}^2)=\{{\bf 0}\}={\bf 0}\mathcal{A}{\bf 0}\cap (\mathcal{S}^2\cup\mathcal{N}^2).
$$

Now we take an element of the form $\mu p_2+a$, $\mu\neq 0$ and
consider the annihilator $^{\perp_q}\{\mu p_2+a\}$. Since $\mathcal{N}$ is an ideal,
by the properties of the Peirce decomposition and (1.2), (1.5), we have
$$
(\mu p_2+a)b^2(\mu p_2 +a)=\mu p_2b^2\mu p_2+
\mu p_2b^2a+ab^2\mu p_2+ab^2a
$$
$$
=\mu p_2b^2\mu p_2={\bf 0}
$$
for every $b\in \mathcal{N}$.

By (1.3), (1.5), We also have
$$
(\mu p_2+a)(\nu p_1+\rho p_2)^2(\mu p_2 +a)
$$
$$
=(\mu p_2+a)(\nu^2 p_1+\rho^2 p_2)^2(\mu p_2+a)
$$
$$
=\mu^2\rho^2 p_2+\rho^2ap_2a+2\rho^2\mu p_2a\neq {\bf 0},
$$
if $\rho\neq 0$. Hence,
$$
^{\perp_q}\{\mu p_2+a\}\cap (\mathcal{S}^2\cup\mathcal{N}^2)=\{\mu^2p_1:\mu\in\mathbb{F}\}\cup\mathcal{N}^2.
$$
Since $p_1\mathcal{A}p_1\cap (\mathcal{S}^2\cup\mathcal{N}^2)=\{p_1,\mathcal{N}^2\}$ by (1.2), we have
$$
^{\perp_q}\{\mu p_2+a\}\cap (\mathcal{S}^2\cup\mathcal{N}^2)=p_1\mathcal{A}p_1\cap (\mathcal{S}^2\cup\mathcal{N}^2).
$$
So, $\mathcal{A}$ is an almost inner Rickart algebra. This ends the proof of the first case in the theorem.

Case $e=p_2$ is the same as case $e=p_1$. In this case, just $p_1$, $p_2$ have been swapped.

If there exists $a\in N$ such that ${^{\perp_q}\{a\}}\cap (\mathcal{S}^2\cap \mathcal{N}^2)=p_1\mathcal{A}p_1\cap (\mathcal{S}^2\cap \mathcal{N}^2)$ or ${^{\perp_q}\{a\}}\cap (\mathcal{S}^2\cap \mathcal{N}^2)=p_2\mathcal{A}p_2\cap (\mathcal{S}^2\cap \mathcal{N}^2)$ we are in some
of the previous cases. Thus, we can assume that, for every $a\in N$,
${^{\perp_q}\{a\}}\cap (\mathcal{S}^2\cap \mathcal{N}^2)=(p_1+p_2)\mathcal{A}(p_1+p_2)\cap (\mathcal{S}^2\cap \mathcal{N}^2)$.

Then, similar to the proof of Theorem \ref{1.4} we have the conditions (2.1),(2.2) and (2.4) are valid.

We prove (2.3).
Let $b$ be another nonzero element from the radical $\mathcal{N}$.
Then, by the assumption,
$$
x^2\in^{\perp_q}\{b\}\cap (\mathcal{S}^2\cup\mathcal{N}^2)=(p_1+p_2)\mathcal{A}(p_1+p_2)\cap (\mathcal{S}^2\cup\mathcal{N}^2), \,\, \, b(p_1+p_2)b={\bf 0}.
$$
Since $p_1+p_2\in f\mathcal{A}f\cap (\mathcal{S}^2\cup\mathcal{N}^2)$, we have
$p_1+p_2\in^{\perp_q}\{b\}$. Hence, $b(p_1+p_2)b={\bf 0}$. Since $b$ is arbitrarily chosen,
for any $i$, $j$, $(e_i+e_j)(p_1+p_2)(e_i+e_j)={\bf 0}$.
Hence, $e_i(p_1+p_2)e_j+e_j(p_1+p_2)e_i={\bf 0}$. So, (2.3) is valid.

Now we prove (2.5) and (2.6).
We have $\mathcal{N}^2\subseteq (p_1+p_2)\mathcal{A}(p_1+p_2)$ by (2.2).
Moreover, $\mathcal{S}^2\cup\mathcal{N}^2\subseteq (p_1+p_2)\mathcal{A}(p_1+p_2)$.
We have, for every $b\in \mathcal{N}$,
\[
{^{\perp_q}\{b\}}\cap (\mathcal{S}^2\cup\mathcal{N}^2)=(p_1+p_2)\mathcal{A}(p_1+p_2)\cap (\mathcal{S}^2\cup\mathcal{N}^2).
\]
Hence, for every $b\in \mathcal{N}$, we have
\[
bp_1b={\bf 0},  bp_2b={\bf 0}
\]
since $\mathcal{S}^2\subseteq (p_1+p_2)\mathcal{A}(p_1+p_2)$, i.e.,
$\mathcal{S}^2\subseteq {^{\perp_q}\{b\}}\cap (\mathcal{S}^2\cup\mathcal{N}^2)$.
Thus, $bp_kb=0$, $k=1,2$, for any $b\in \mathcal{N}$.
From this it follows that
$$
e_ip_ke_j+e_jp_ke_i={\bf 0}, \,\,\, k=1,2, \,\,\, i,j=1,2,\ldots,n.          \eqno{(2.6)}
$$

Now we take the annihilator $^{\perp_q}\{p_1\}$. Then there exists an idempotent $e\in\mathcal{S}$ such that
$$
^{\perp_q}\{p_1\}\cap (\mathcal{S}^2\cup\mathcal{N}^2)=e\mathcal{A}e\cap (\mathcal{S}^2\cup\mathcal{N}^2).
$$
Since $p_2\in ^{\perp_q}\{p_1\}$, we have $e=p_2$.
We have $\mathcal{N}^2\subseteq (p_1+p_2)\mathcal{A}(p_1+p_2)$.
From this it follows that, for any $x\in\mathcal{N}$,
\[
x^2=p_1x^2p_1+p_1x^2p_2+p_2x^2p_1+p_2x^2p_2.
\]
If $p_1x^2p_1=0$ then $x^2=p_2x^2p_2$, i.e.,
$p_1x^2p_2+p_2x^2p_1=0$. Thus, for any $x\in\mathcal{N}$, if $p_1x^2p_1=0$ then $p_1x^2p_2+p_2x^2p_1=0$.

Similarly, for any $x\in\mathcal{N}$, if $p_2x^2p_2=0$ then $p_1x^2p_2+p_2x^2p_1=0$.

Therefore, for any $x\in\mathcal{N}$ such that $x^2\neq 0$, one of the following three cases is valid:
1) if $p_1x^2p_1=0$ then $p_1x^2p_2+p_2x^2p_1=0$, 2) if $p_2x^2p_2=0$ then $p_1x^2p_2+p_2x^2p_1=0$
and 3) $p_1x^2p_1\neq 0$ and $p_2x^2p_2\neq 0$. From these it follows that (2.5) is valid.

Now let us prove that an associative algebra $\mathcal{A}$ with a basis $\{p_1,p_2,e_1,e_2,\dots,e_n\}$,
satisfying the properties (2.1)--(2.5), is an almost inner Rickart algebra.

Now we take the annihilator $^{\perp_q}\{p_1\}$. We prove that
$$
^{\perp_q}\{p_1\}\cap (\mathcal{S}^2\cup\mathcal{N}^2)=p_2\mathcal{A}p_2\cap (\mathcal{S}^2\cup\mathcal{N}^2).
$$
Clearly, $p_2\mathcal{A}p_2\cap (\mathcal{S}^2\cup\mathcal{N}^2)\subseteq ^{\perp_q}\{p_1\}\cap (\mathcal{S}^2\cup\mathcal{N}^2)$ and
$^{\perp_q}\{p_1\}\cap \mathcal{S}^2\subseteq p_2\mathcal{A}p_2\cap \mathcal{S}^2$. It is sufficient to prove
$^{\perp_q}\{p_1\}\cap \mathcal{N}^2\subseteq p_2\mathcal{A}p_2\cap \mathcal{N}^2$.
We take $x^2\in^{\perp_q}\{p_1\}\cap \mathcal{N}^2$. Then $p_1x^2p_1=0$ and
$p_1x^2p_2+p_2x^2p_1=0$ by (2.5). Hence, $x^2=p_2x^2p_2$ and $x^2\in p_2\mathcal{A}p_2\cap \mathcal{N}^2$.
So, $^{\perp_q}\{p_1\}\cap \mathcal{N}^2\subseteq p_2\mathcal{A}p_2\cap \mathcal{N}^2$.
Therefore, $^{\perp_q}\{p_1\}\cap (\mathcal{S}^2\cup\mathcal{N}^2)=p_2\mathcal{A}p_2\cap (\mathcal{S}^2\cup\mathcal{N}^2)$.

Similarly, we have
$$
^{\perp_q}\{p_2\}\cap (\mathcal{S}^2\cup\mathcal{N}^2)=p_1\mathcal{A}p_1\cap (\mathcal{S}^2\cup\mathcal{N}^2).
$$

Now we take an element of the form $\lambda p_1+\mu p_2+a$, $\lambda\neq 0$ or $\mu\neq 0$,
and $b\in \mathcal{N}$ such that $b^2\neq 0$.
Since $\mathcal{N}$ is an ideal,
by the properties of the Peirce decomposition, we have
$$
(\lambda p_1+\mu p_2+a)b^2(\lambda p_1+\mu p_2+a)
$$
$$
=(\lambda p_1+\mu p_2)b^2(\lambda p_1+\mu p_2)+
(\lambda p_1+\mu p_2)b^2a+ab^2(\lambda p_1+\mu p_2)+ab^2a
$$
$$
=(\lambda p_1+\mu p_2)b^2(\lambda p_1+\mu p_2)\neq {\bf 0}
$$
since $b^2\in (p_1+p_2)\mathcal{A}(p_1+p_2)$. We also have
$$
(\lambda p_1+\mu p_2+a)(\nu p_1+\rho p_2)^2(\lambda p_1+\mu p_2+a)=
$$
$$
(\lambda p_1+\mu p_2+a)(\nu^2 p_1+\rho^2 p_2)(\lambda p_1+\mu p_2+a)
$$
$$
=\lambda^2\nu^2 p_1+\mu^2\rho^2 p_2+c,
$$
where $c\in \mathcal{N}$ since $\mathcal{N}$ is an ideal. Therefore
$$
(\lambda p_1+\mu p_2+a)(\nu p_1+\rho p_2)^2(\lambda p_1+\mu p_2+a)\neq {\bf 0}
$$
if $\nu\neq 0$ or $\rho\neq 0$. Hence,
$$
^{\perp_q}\{\lambda p_1+\mu p_2+a\}\cap (\mathcal{S}^2\cup\mathcal{N}^2)=\{{\bf 0}\}={\bf 0}\mathcal{A}{\bf 0}\cap (\mathcal{S}^2\cup\mathcal{N}^2).
$$
So, $\mathcal{A}$ is an almost inner Rickart algebra. This ends the proof.
\end{proof}

\begin{remark}
Note that, in the case $e=p_1$ of the proof of Theorem \ref{2},
since $A$ is an almost inner Rickart algebra and $p_2\in {^{\perp_q}\{p_1\}}\cap (\mathcal{S}^2\cup\mathcal{N}^2)=
f\mathcal{A}f\cap (\mathcal{S}^2\cup\mathcal{N}^2)$
for some idempotent $f\neq {\bf 0}$ in $\mathcal{A}$ we have $f\in\mathcal{S}^2$ and $p_1fp_1={\bf 0}$. Hence, $p_1f={\bf 0}$ since $\mathcal{S}$
is commutative. So, $f=p_2$ and
\[
{^{\perp_q}\{p_1\}}\cap (\mathcal{S}^2\cup\mathcal{N}^2)=p_2\mathcal{A}p_2\cap (\mathcal{S}^2\cup\mathcal{N}^2).
\]
\end{remark}

The following proposition follows from the Jordan form of a
$3\times 3$ complex matrix.

\begin{proposition}  \label{333}
The associative algebra $M_3(\mathbb{C})$ does not have nilpotent
elements $a\in M_3(\mathbb{C})$ such that $a^3\neq {\bf 0}$.
\end{proposition}

\begin{remark}
If a finite-dimensional semisimple associative algebra is an almost inner Rickart algebra,
then this algebra is an inner RJ-algebra with respect to the Jordan multiplication.
By Proposition \ref{333} and the example on page 8 of \cite{AKh}
the condition of Lemma \ref{1.1} is not sufficient for an associative
algebra $\mathcal{A}$ to be an almost inner Rickart algebra.
\end{remark}

In the following tables associative algebras of dimension two and three
are listed. It is indicated that, wether an associative algebras is an almost inner Rickart algebra
or not for every associative algebra of the tables.

The fourth column of the tables indicates whether the appropriate associative algebra is
an almost inner Rickart algebra or not,
i.e., if yes, then sign "+" is put in the appropriate place of the column,
if not, then sign "-" is put in this place. All
notations of the tables 1 and 2 are taken from \cite{RRB}.


\begin{small}
\begin{center}
{\bf Table 1}

{\bf Two-dimensional associative algebras.}
\end{center}

\begin{center}
\begin{tabular}{|c|c|c|c|c|}
  \hline
  $J$ & Multiplication & Is this    & According to  \\
      &  table         & algebra an  &  which theorem?  \\
      &                & almost inner Rickart algebra?  &   \\

\hline
  $\mathcal{A}_2^1$    &  $n_1n_1=n_2$ & - & Theorem 2.4 \\
\hline
  $\mathcal{A}_2^2$    &  $e_1e_1=e_1$, & + & Theorem 2.5 \\
                       &  $e_1n_1=n_1$ &   &               \\
  \hline
  $\mathcal{A}_2^3$    & $e_1e_1=e_1$,   & + &  Theorem 2.5 \\
                       & $n_1e_1=n_1$   &   &                 \\
\hline
$\mathcal{A}_2^4$      & $e_1e_1=e_1$,   & + &  Theorem 2.5 \\
                       &$e_1n_1=n_1$, $n_1e_1=n_1$ &   &          \\
\hline
\end{tabular}
\end{center}
\end{small}

\bigskip

\newpage

\begin{small}
\begin{center}
{\bf Table 2}

{\bf Three-dimensional associative algebras.}
\end{center}

\begin{center}
\begin{tabular}{|c|c|c|c|c|}
  \hline
  $J$ & Multiplication & Is this           & According to      \\
      &  table         & algebra an inner  &  which theorem?   \\
      &                & Rickart algebra?     &                   \\

\hline
  $\mathcal{A}_3^1$         & $n_1n_3=n_2$, $n_3n_1=n_2$ & - & Theorem 2.4   \\
\hline
  $\mathcal{A}_3^2(\alpha)$ &  $n_1n_3=n_2$              & - & Theorem 2.4   \\
                            &$n_3n_1=\alpha n_2, \alpha\in \mathbb{C}\setminus\{1,-1\}$& &  \\
\hline
  $\mathcal{A}_3^2(\alpha)$ &  $n_1n_3=n_2$                & + & Theorem 2.4   \\
                            &$n_3n_1=\alpha n_2, \alpha=-1$ &   &  \\
\hline
  $\mathcal{A}_3^3$    & $n_1n_1=n_2$, $n_1n_2=n_3$  & - &  Theorem 2.4    \\
                       & $n_2n_1=n_3$                &   &                 \\
\hline
$\mathcal{A}_3^4$      & $n_1e_1=n_2$, $n_2e_1=n_2$ & + &   Theorem 2.5    \\
                       &$e_1e_1=e_1$                &   &                  \\
\hline
$\mathcal{A}_3^5$      & $n_2e_1=n_2$, $e_1n_1=n_1$  & + &  Theorem 2.5 \\
                       &$e_1e_1=e_1$                 &   &              \\
\hline
$\mathcal{A}_3^6$      & $e_1n_1=n_2$, $e_1n_2=n_2$ & + & Theorem 2.5 \\
                       & $e_1e_1=e_1$               &   &             \\
\hline
$\mathcal{A}_3^7$    & $n_1e_1=n_1$, $e_1e_1=e_1$ & +   & Theorem 2.8 \\
                     & $e_2n_1=n_1$, $e_2e_2=e_2$ &     & both case 1 and case2     \\
\hline
  $\mathcal{A}_3^8$  &  $n_1e_1=n_1$, $n_2e_1=n_2$ & + & Theorem 2.5 \\
                     &  $e_1n_1=n_1$, $e_1e_1=e_1$ &   &             \\
  \hline
  $\mathcal{A}_3^9$  & $n_2e_1=n_2$, $e_1n_1=n_1$  & + &  Theorem 2.5 \\
                     & $e_1n_2=n_2$, $e_1e_1=e_1$  &   &               \\
\hline
$\mathcal{A}_3^{10}$ & $n_1e_1=n_1$, $n_2e_1=n_2$               & + & Theorem 2.5 \\
                     & $e_1n_1=n_1$, $e_1n_2=n_2$, $e_1e_1=e_1$ &   &              \\
\hline
$\mathcal{A}_3^{11}$ &  $n_1e_1=n_2$, $n_2e_1=n_2$               & + & Theorem 2.5  \\
                     &  $e_1n_1=n_2$, $e_1n_2=n_2$, $e_1e_1=e_1$ &   &               \\
\hline
  $\mathcal{A}_3^{12}$  & $n_1n_1=n_2$, $n_1e_1=n_1$, $n_2e_1=n_2$  & - &  Theorem 2.5 \\
                        &$e_1n_1=n_1$, $e_1n_2=n_2$, $e_1e_1=e_1$   &   &              \\
\hline
\end{tabular}
\end{center}
\end{small}

\medskip

We are sincerely grateful to the reviewers for their careful reading of our article and valuable comments. It should be noted that, thanks to the reviewers, a number of significant shortcomings of the article were corrected. We would like to express our gratitude to one of the reviewers who worked diligently to identify errors from the beginning of the article to its end. Thanks to that reviewer, the main results and their proofs have been corrected.

\end{document}